\documentclass[12pt,a4paper,reqno]{amsart}
\usepackage[utf8]{inputenc}
\usepackage[T1]{fontenc}
\usepackage{amsmath}
\usepackage{amsthm}
\usepackage{amssymb}
\usepackage[abbrev]{amsrefs}
\usepackage{mathrsfs}
\usepackage[dvipsnames,hyperref]{xcolor}
\usepackage{bm}
\usepackage{bbm}
\usepackage{enumitem}
\usepackage{hyperref}
\AtBeginDocument{\def\MR#1{}}
\makeatletter
\@namedef{subjclassname@2020}{%
\textup{2020} Mathematics Subject Classification}
\makeatother
\newtheorem{clm}{}

\newtheorem{claim}[clm]{Claim}
\newtheorem{thm}{}[section]
\newtheorem{theorem}[thm]{Theorem}
\newtheorem{corollary}[thm]{Corollary}
\newtheorem{lemma}[thm]{Lemma}
\newtheorem{proposition}[thm]{Proposition}
\newtheorem{question}[thm]{Question}
\theoremstyle{definition}

\theoremstyle{remark}
\newtheorem{remark}[thm]{Remark}
\numberwithin{equation}{section}
\allowdisplaybreaks
\newcommand{\littletaller}{\mathchoice{\vphantom{\big|}}{}{}{}}
\newcommand\restr[2]{{\left.\kern-\nulldelimiterspace #1 \littletaller \right|_{#2}}}
\newcommand{\Nnorm}[1]{{\left\vert\kern-0.25ex\left\vert\kern-0.25ex\left\vert #1
\right\vert\kern-0.25ex\right\vert\kern-0.25ex\right\vert}}
\newcommand{\abs}[1]{\left\lvert#1\right\rvert}
\newcommand{\norm}[1]{\left\lVert#1\right\rVert}
\newcommand{\enbrace}[1]{\left\lbrace#1\right\rbrace}
\newcommand{\enbrak}[1]{\left[#1\right]}
\newcommand{\enpar}[1]{\left(#1\right)}
\newcommand{\BV}{\ensuremath{\mathrm{BV}}}
\newcommand{\Dt}{\ensuremath{\mathcal{D}}}
\newcommand{\St}{\ensuremath{\mathcal{S}}}
\newcommand{\Bt}{\ensuremath{\mathcal{B}}}
\newcommand{\Rt}{\ensuremath{\mathcal{D}}}
\newcommand{\Ts}{\ensuremath{\mathcal{T}}}
\newcommand{\Ft}{\ensuremath{\mathcal{F}}}
\newcommand{\Pt}{\ensuremath{\mathcal{P}}}
\newcommand{\Ct}{\ensuremath{\mathcal{C}}}
\newcommand{\Nt}{\ensuremath{\mathcal{N}}}
\newcommand{\Gt}{\ensuremath{\mathcal{G}}}
\newcommand{\Ht}{\ensuremath{\mathcal{H}}}
\newcommand{\XB}{\ensuremath{\mathcal{X}}}
\newcommand{\YB}{\ensuremath{\mathcal{Y}}}
\newcommand{\Ind}{\ensuremath{\mathbbm{1}}}
\newcommand{\EE}{\ensuremath{\mathbb{E}}}
\newcommand{\FF}{\ensuremath{\mathbb{F}}}
\newcommand{\NN}{\ensuremath{\mathbb{N}}}
\newcommand{\RR}{\ensuremath{\mathbb{R}}}
\newcommand{\XX}{\ensuremath{\mathbb{X}}}
\newcommand{\YY}{\ensuremath{\mathbb{Y}}}
\newcommand{\Sym}{\ensuremath{\mathbb{S}}}
\newcommand{\xx}{\ensuremath{\bm{x}}}

\newcommand{\prim}{\ensuremath{\bm{\sigma}}}
\newcommand{\primt}{\ensuremath{\bm{\tau}}}
\newcommand{\usdf}{\ensuremath{\bm{\varphi_u}}}
\newcommand{\lsdf}{\ensuremath{\bm{\varphi_l}}}
\newcommand{\SNA}{\ensuremath{\bm{S}}}
\newcommand{\SNC}{\ensuremath{\bm{T}}}
\newcommand{\SNB}{\ensuremath{\bm{T^c}}}
\newcommand{\unc}{\ensuremath{\bm{k}}}

\DeclareMathOperator{\supp}{supp}

\begin{document}
\title[]{Isometric Renormings for Greedy Bases in Banach Spaces, with applications to the Haar system in $\bm{L_{p}[0,1]}$, $\bm{1<p<\infty}$.}
\author[Albiac]{Fernando Albiac}
\address{Fernando Albiac\\ Department of Mathematics, Statistics, and Computer Sciencies--InaMat2\\
Universidad P\'ublica de Navarra\\
Campus de Arrosad\'{i}a\\
Pamplona\\
31006 Spain}
\email{fernando.albiac@unavarra.es}
\author[Ansorena]{Jos\'e L. Ansorena}
\address{José Luis Ansorena\\ Department of Mathematics and Computer Sciences\\
Universidad de La Rioja\\
Logro\~no\\
26004 Spain}
\email{joseluis.ansorena@unirioja.es}
\author[Berasategui]{Miguel Berasategui}
\address{Miguel Berasategui\\
UBA - Pab I, Facultad de Ciencias Exactas y Naturales\\
Universidad de Buenos Aires\\
(1428), Buenos Aires, Argentina}
\email{mberasategui@dm.uba.ar}
\author[Bern\'a]{Pablo M. Bern\'a}
\address{Pablo M. Bern\'a\\
Departamento de Matemáticas, CUNEF Universidad\\
Madrid, 28040 Spain}
\email{pablo.berna@cunef.edu}
\begin{abstract}
We investigate the problem of improving the greedy-type constant of a basis by means of an equivalent renorming of the ambient Banach space. Our main result shows that if a Banach space admits an unconditional and bidemocratic basis whose fundamental function satisfies certain regularity properties, then the space can be renormed so that the basis becomes isometrically greedy. The renorming simultaneously ensures lattice $1$-unconditionality, isometric bidemocracy, and allows prescribing the fundamental function up to a suitable regularization. As a principal application, we resolve a long-standing problem posed by Albiac--Wojtaszczyk in 2006 by proving that for each $1<p<\infty$ the $L_p$-normalized Haar system can be made $1$-greedy under an equivalent norm of $L_p$. Further applications include isometric greedy renormings for bases of Besov spaces, mixed-norm direct sums, and for a wide class of subsymmetric and conditional bases, including spreading models and the canonical basis of Schlumprecht space. These results show that isometric greedy renormings arise in far greater generality than previously known.
\end{abstract}
\thanks{F.\@ Albiac and J.\@ L.\@ Ansorena acknowledge the support of the Spanish Ministry for Science and Innovation under Grant PID2022-138342NB-I00 for \emph{Functional Analysis Techniques in Approximation Theory and Applications.}}
\thanks{M.\@ Berasategui is supported by Grants CONICET PIP11220200101609CO y ANPCyT PICT 2018-04104 (Consejo Nacional de Investigaciones Científicas y Técnicas and Agencia Nacional de Promoción de la Investigación, el Desarrollo Tecnológico y la Innovación, Argentina)}
\thanks{Pablo M.\@ Berná is supported by Grant PID2022-142202NB-I00 (Agencia Estatal de Investigación, Spain)}
\dedicatory{Dedicated to Prof.\@ Przemek Wojtaszczyk, mentor, friend, and co-parent of the isometric greedy basis theory, for many helps.}
\maketitle
\section{Introduction and Background}\noindent
The theory of greedy approximation lies at the crossroads of nonlinear approximation, Banach space geometry, and the structure of bases. Given a Banach space $\XX$ (over the real or complex field $\FF$) with a complete biorthogonal system $(\xx_n,\xx_n^*)_{n\in \Nt}$, and an element $f\in \XX$ with formal expansion $f=\sum_{n\in\Nt} \xx_n^*(f)\, \xx_n$, the \emph{Thresholding Greedy Algorithm} (TGA for short) selects, for each $m\in \NN$ a set $A\subset\NN$ of cardinality $m$ consisting of indices corresponding to the $m$ largest coefficients $\abs{\xx_n^*(f)}$. The associated $m$-term \emph{greedy approximant} is
\[
G_m(f)=\sum_{n\in A} \xx_n^*(f)\, \xx_n.
\]
A basis $(\xx_n)_{n\in \Nt}$ is called \emph{$C$-greedy}, $1\le C<\infty$, if the greedy approximants are always within a multiplicative factor $C$ of the optimal sparse approximation error, that is, $\norm{f-G_m(f)}\le C\,\sigma_m(f)$ for all $f\in\XX$ and $m\in\NN$, where
\[
\sigma_m(f)=\inf\enbrace{\norm{f-\sum_{n\in A} a_n\, \xx_n} \colon \abs{A}=m, \, a_n\in\FF}.
\]

The systematic study of greedy bases began with the seminal paper of Konyagin and Temlyakov~\cite{KoTe1999}, where they introduced democracy and superdemocracy as quantitative symmetry conditions on bases and proved that these concepts capture the essence of greedy approximation. Their main result states that a basis is greedy if and only if it is unconditional and democratic. In particular, unconditionality alone is not enough: the uniform distribution of coefficient mass across index sets of equal cardinality is just as essential as stability under coordinate suppressions.

The above characterization is sharp but does not address the delicate quantitiative case. The class of $C$-greedy bases becomes increasingly symmetric as $C$ approaches $1$, and studying the boundary case $C=1$ is of special importance.

In their foundational work~\cite{AW2006}, the authors introduced \emph{Property~(A)}. This property asserts that replacing the signs or permuting the coefficients among a collection of ``large'' coordinates (in modulus) does not increase the norm. Their main theorem is the characterization of $1$-greedy bases as those bases that are $1$-suppression unconditional and satisfy Property~(A). Thus $1$-greedy bases enjoy a symmetry akin to that of symmetric bases, but restricted to the coordinates that dominate the coefficient vector. As a consequence, the structure of $1$-greedy bases is qualitatively more rigid than that of general greedy bases.

Albiac--Wojtaszczyk result provided a satisfactory positive answer to a problem raised by Wojtaszczyk (see \cite{Wojt2003}*{Problem 1}) and was the springboard for the, so called, \emph{isometric theory of greedy bases}. They also posed several open problems concerning the existence of renormings that improve greedy and democratic constants. Most notably, they asked whether the $L_p$-normalized Haar system, $1<p<\infty$, can be made $1$-greedy under an equivalent norm of the classical Lebesgue space $L_p$ (see \cite{AW2006}*{Problem 6.2}). Of course, the case $p=2$ is naturally excluded because the normalized Haar system of $L_2$ is orthonormal, hence $1$-greedy.

The first thorough, in-depth approach to greedy renormings was carried out in the papers \cites{DOSZ2011,DKOSZ2014}. Their contributions clarified both the limitations and possibilities of renorming procedures. On the negative side:

\begin{theorem}[\cite{DOSZ2011}*{Corollary 1.3}]
There is no renorming of the dyadic Hardy space $H_1$ that makes the Haar basis $1$-greedy.
\end{theorem}

\begin{theorem}[\cite{DOSZ2011}*{Corollary 1.2}]
Tsirelson's space admits no renorming with any normalized $1$-greedy basis.
\end{theorem}

Thus, $1$-greedy renormings cannot be expected universally. On the positive side, the authors showed:

\begin{theorem}[\cite{DKOSZ2014}*{Theorem 10}]
Any greedy basis can be renormed to become lattice $1$-unconditional and $(2+\varepsilon)$-greedy for any $\epsilon>0$.
\end{theorem}

\begin{theorem}[\cite{DOSZ2011}*{Theorem 2.1} and \cite{DKOSZ2014}*{Theorem 4}]\label{thm:DKOSZ:4}
If a greedy basis is also bidemocratic, it can be renormed to become lattice $1$-unconditional, isometrically bidemocratic, and $(1+\varepsilon)$-greedy for any $\epsilon>0$.
\end{theorem}

\begin{theorem}[\cite{DKOSZ2014}*{Corollary 12}]
For all $\epsilon>0$ there is an equivalent norm on $H_1$ and on Tsirelson's space $\Ts$ such that the Haar system, respectively, the unit vector basis, is normalized, $1$-unconditional, and $(1+\epsilon)$-greedy for any $\epsilon>0$.
\end{theorem}

These results represent the strongest known general renorming statements prior to this work. Later developments produced further positive examples: most notably, Garling sequence spaces admit renormings under which their canonical bases are $1$-greedy \cite{AAW2018b}*{Section 3}. Related to this, new work has also produced examples of bases satisfying Property~(A) yet failing to be $1$-suppression unconditional~\cite{AABCO2024}, thus answering another important open problem in the theory (see \cite{AlbiacAnsorena2017b}*{Problem 4.4}).

The central contribution of this paper is a renorming scheme that achieves the isometric theory. Our main theorem states that if a Banach space $\XX$ has a greedy basis whose fundamental function satisfies both the lower regularity property and the upper regularity property, then $\XX$ admits an equivalent norm such that the basis becomes $1$-greedy, lattice $1$-unconditional, and isometrically bidemocratic. Moreover, the new fundamental function can be prescribed up to Dini regularization (a sharp form of regularity that will be explained below). A direct application solves the long-standing open problem of Albiac--Wojtaszczyk: for all $1<p<\infty$, the $L_p$-normalized Haar system can be made isometrically greedy under a renorming of $L_p$. Our renorming method extends to a wide range of settings, which demonstrate that isometrically greedy renormings are surprisingly widespread.

The paper is organized as follows. Section~\ref{Sec:Prelim} collects the terminology and background needed throughout the article, including democracy and superdemocracy, Property~(A), quasi-greedy and almost greedy bases, bidemocracy, regularity properties of fundamental functions, and duality estimates for greedy-type systems. Section~\ref{sect:construction} develops the unified renorming scheme that underlies all subsequent results. This construction is built from expressions of Marcinkiewicz type, quasi-greedy embeddings, and Dini regularizations. Section~\ref{Sect:Renorm} applies this machinery to greedy bases, yielding the main isometric renorming theorem and several structural consequences under type or cotype assumptions. In Subsection~\ref{Sec:Haar} we specialize the theory to the Haar system in $L_p$, providing background on its geometry and concluding with the solution to the long-standing question of whether it can be made $1$-greedy under an equivalent renorming. Subsection~\ref{Sect:Besov} turns to further applications, including Besov spaces and mixed-norm direct sums, and shows that entire classes of function space models admit isometric greedy renormings under mild hypotheses. Section~\ref{Sect:Subsym} treats subsymmetric bases, establishing new examples, such as Schlumprecht's space, of renormings achieving isometric subsymmetry and greedy behavior. Finally, Section~\ref{Sect:Condit} deals with renormings of spaces with conditional bases, and Section~\ref{Sect:Prob} closes with several open problems.
\section{Preliminaries}\label{Sec:Prelim}\noindent
The remainder of the paper will rely frequently on a small collection of notation and auxiliary concepts that we record in his section for ease of reference. The theory of greedy and greedy-type bases has evolved into a broad and sophisticated subject, and it would be neither practical nor desirable to attempt to summarize all of its developments here. Instead, we have focused on the notions most relevant for the renorming scheme developed later in the paper. For a comprehensive exposition of the classical and nonlinear approximation aspects of the subject, we refer the reader to \cite{AlbiacKalton2016}*{Chapter 10}; and for a detailed and up-to-date account of the modern theory, including open problems and structural refinements, we recommend \cite{AABW2021}. We now introduce the standing notation that will be most heavily used throughout.

Given $\alpha=(a_n)_{n\in\Nt}$ and $\beta=(b_n)_{n\in\Nt}$ in $[0,\infty)$, we say that $\alpha$ \emph{$C$-domintates} $\beta$ if $ b_n \le C a_n$ for all $n\in\Nt$. If the family $\alpha$ $C$-dominates the family $\beta$ and $\beta$ $C$-dominates $\alpha$, we say that $\alpha$ and $\beta$ are \emph{$C$-equivalent}. If there is $C\in(0,\infty)$ such that $\alpha$ $C$-domintates $\beta$, we we put $\alpha\lesssim\beta$ and say that $\alpha$ dominates $\beta$. Similarly, we say that $\alpha$ and $\beta$ are equivalent, and we put $\alpha\approx\beta$, if $\alpha\lesssim \beta \lesssim \alpha$.

Given a set $\Nt$, we will denote by $\Pt_{<\infty}(\Nt)$ the set of all finite subsets of $\Nt$. Given $m\in\NN$, $\Pt_{\le m}(\Nt)$ (resp., $\Pt_{\ge m}(\Nt)$) stands for the set of all finite subsets of $\Nt$ of cardinality at most (resp., at least) $m$. By $c_{00}(\Nt)$ we denote the space of all eventually null functions from $\Nt$ to $\FF$, and by $c_{0}(\Nt)$ the set of all functions $f\colon\Nt\to \FF$ such that $\lim_{n\in\Nt} f(n)=0$.

We will denote by $S_\XX$ and $B_\XX$ the unit sphere and the closed unit ball, respectively, of a Banach space $\XX$. In the particular case of the field of scalars we put $\EE=S_\FF$. Given a collection of vectors $\XB=(\xx_n)_{n\in\Nt}$ in $\XX$, a subset $A\in\Pt_{<\infty}(\Nt)$, and $\varepsilon\in\EE^A$ we put
\[
\Ind_{\varepsilon,A}[\XB,\XX]=\sum_{n\in A} \varepsilon_n \, \xx_n.
\]

Given a subset of vectors $W\subset \XX$, the symbol $\enbrak{W}$ stands for its closed linear span.

Suppose that $\XB=(\xx_n)_{n\in\Nt}$ is a \emph{complete minimal system} in $\XX$, i.e., $\enbrak{\XB}=\XX$, and there is a family $\XB^*=(\xx_n^*)_{n\in\Nt}$ in $\XX^{\ast}$ such that
\[
\xx_n^*(\xx_k)=\delta_{n,k}, \quad n,k\in\Nt.
\]
In this case, the family $\XB^*$ is unique, and it is a minimal system itself. If fact, we can identify $\XB^{**}$ with $\XB$ via the canonical map from $\XX$ into $\enbrak{\XB^*}^*$. We will call $\XB^*$ the \emph{dual system} of $\XB$. The \emph{support} of $f\in\XX$ relative to $\XB$ will be the set
\[
\supp(f)=\enbrace{n\in\NN \colon \xx_n^*(f)\not=0}.
\]

When we use the term \emph{basis} we will refer to a complete minimal system such that both $\XB$ and $\XB^*$ are bounded, so that in particular $\XB$ is \emph{semi-normalized}, i.e.,
\[
0<\inf_{n\in\Nt} \norm{\xx_n}, \quad \sup_{n\in\Nt} \norm{\xx_n}<\infty,
\]
and \emph{$M$-bounded}, i.e., $\sup_n \norm{\xx_n} \norm{\xx_n^*}<\infty$.

Often, knowing that a space has a basis, however, is not sufficient and one needs to consider some special types bases, whose definition and main properties we gather below for self-reference.

A sequence $\XB=(\xx_n)_{n\in\Nt}$ in a Banach space $\XX$ is an \emph{unconditional basis} if for all $f\in\XX$ there are unique scalars $(a_n)_{n\in\Nt}$ is $\FF$ such that the series $\sum_{n\in\Nt} a_n\, \xx_n$ converges to $f$ unconditionally. If $\XB$ is an \emph{unconditional basis}, in particular $\XB$ is a complete minimal system and there is a constant $C$ such that \begin{equation}\label{neweq}
\norm{M_\lambda}\le C\norm{\lambda}_\infty\quad \text{for all}\; \lambda\in c_{00}(\Nt),
\end{equation}
where
\[
M_\lambda[\XB,\XX] \colon \XX\to \XX,\quad \xx_n\mapsto \lambda_n \, \xx_n, \quad n\in\Nt.
\]

If $\lambda=\chi_A$ for some $A\in\Pt_{<\infty}(\Nt)$, we set $S_A[\XB,\XX]=M_\lambda[\XB,\XX]$. If the complete minimal system and the space under discussion are clear from context, we simply put $S_A=S_A[\XB,\XX]$ and $M_\lambda=M_\lambda[\XB,\XX]$. If inequality \eqref{neweq} holds for some constant $C$, we say that $\XB$ is \emph{lattice $C$-unconditional} and we will denote by $K_u[\XB]$ the smallest such constant. If $\XB$ is lattice $C$-unconditional, then the linear operator $M_\lambda$ is well-defined for all $\lambda\in\ell_\infty$, and $\norm{M_\lambda}\le C$. If $C\in[1,\infty)$ is such that $\norm{S_A[\XB,\XX]}\le C$ for all $A\subset\Nt$ finite, we say that $\XB$ is \emph{suppression $C$-unconditional}. If $\XB$ is suppression $C$-unconditional for some $C$, then $\XB$ is unconditional. Any unconditional basis becomes lattice $1$-unconditional under a straightforward renorming of the space.

In order to quantify how far a complete minimal system $\XB$ is from being an unconditional basis, in approximation theory we use the sequence of \emph{unconditionality parameters}
\[
\unc_{m}=\unc[\XB,\XX](m) =\sup_{A\in\Pt_{\le m}(\Nt)} \norm{S_A[\XB,\XX]}, \quad m\in\NN.
\]
Note that $\XB$ is unconditional if and only if $(\unc_{m})_{m=1}^\infty$ is bounded.

Let $\beta\colon\NN\to\NN$ be an increasing map, $\XX$ be a Banach space, and $(\xx_n)_{n=1}^\infty$ be complete minimal system modelled over the positive integers. Then $\XB$ is equivalent to its subsequence $(\xx_{\beta(n)})_{n=1}^\infty$ if and only if the right shift
\[
\quad \xx_n \mapsto \xx_{\beta(n)}, \, n\in\NN,
\]
extends to an isomorphic embedding from $\XX$ into $\XX$. A sequence $\XB=(\xx_n)_{n=1}^\infty$ in $\XX$ is said to be a \emph{subsymmetric basis} of $\XX$ if it is an unconditional basis equivalent to all its subsequences. In that case, we define
\[
L_\beta\enpar{\sum_{n=1}^{\infty}a_n\xx_n}=\sum_{n=1}^{\infty}a_n \xx_{\beta(n)}
\]
for each increasing map $\beta\colon\NN\rightarrow \NN.$ If $\XB$ is subsymmetric, then there is a constant $C\in[0,\infty)$ such that
\[
\max\enbrace{ \norm{M_\lambda \circ L_\beta}, \norm{\enpar{M_\lambda \circ L_\beta}^{-1}}}\le C
\]
for all $\lambda\in\EE^\NN$ and all increasing maps $\beta\colon\NN\to\NN$ \cite{Ansorena2018}*{Corollary 3.9}. If the above inequality holds with a particular $C$, we say that $\XB$ is $C$-subsymmetric, and $1$-subsymmetric bases will be called \emph{isometrically subsymmetric}. It is known that any subsymmetric basis is isometrically subsymmetric under a suitable renorming of the space \cite{Ansorena2018}*{Theorem 1.2}.

We consider the following inequality associated with $C\in[1,\infty)$, $A$, $B\in\Pt_{<\infty}(\Nt)$, $\varepsilon\in\EE^A$, $\delta\in\EE^B$ and $g\in\XX$:
\begin{equation}\label{eq:SSL}
\norm{\Ind_{\varepsilon,A}[\XB,\XX]+g}\le C \norm{\Ind_{\delta,B}[\XB,\XX]+g}.
\end{equation}
If $\XB$ is a complete minimal system with dual system $(\xx_n^*)_{n\in\Nt}$ and \eqref{eq:SSL} holds for all $A$, $B\in \Pt_{<\infty}(\Nt)$ with $\abs{A} \le \abs{B}$ and $A\cap B=\emptyset$, all $\varepsilon\in \EE^A$, all $\delta\in\EE^B$, and all $g\in\XX$ with $\supp(g) \subset \Nt \setminus(A\cup B)$ and $\abs{\xx_n^*(g)}\le 1$ for all $n\in\Nt$, we say that $\XB$ is \emph{$C$-symmetric for largest coefficients}. If \eqref{eq:SSL} holds with $g=0$ for all $A$, $B\in \Pt_{<\infty}(\Nt)$ with $\abs{A} \le \abs{B}$, all $\varepsilon\in \EE^A$ and all $\delta\in\EE^B$, we say that $\XB$ is $C$-\emph{super-democratic}. If we also impose that $\varepsilon$ and $\delta$ are constant, we say that $\XB$ is \emph{$C$-democratic}. By a standard convexity argument, replacing $\abs{A} \le \abs{B}$ with $\abs{A}=\abs{B}$ in the definition of $C$-\emph{super-democracy} or $C$-symmetry for largest coefficients results in the same property.

In all cases, if the constant $C$ in \eqref{eq:SSL} is irrelevant, we drop it from the notation. If $C=1$, we say that the corresponding property property holds isometrically. Symmetry for largest coefficients implies superdemocracy, superdemocracy implies democracy, and democracy entails semi-normalization.

We define the \emph{lower superdemocracy} and the \emph{upper superdemocracy function}, also known as \emph{fundamental function}, of $\XB$ by
\begin{align*}
\lsdf(m)=\lsdf[\XB,\XX](m)&=\inf\enbrace{ \norm{\Ind_{\varepsilon,A}[\XB,\XX]} \colon A\in\Pt_{\ge m}(\Nt),\, \varepsilon\in\EE^A},\\
\usdf(m)=\usdf[\XB,\XX](m)&=\sup\enbrace{ \norm{\Ind_{\varepsilon,A}[\XB,\XX]} \colon A\in\Pt_{\le m}(\Nt), \, \varepsilon\in\EE^A},
\end{align*}
respectively, for all $m\in\NN$. In this terminology, $\XB$ is $C$-superdemocratic if and only $\usdf(m)\le C \lsdf(m)$ for all $m\in\NN$.

Given a sequence $\prim\colon\NN\to(0,\infty)$ we define its \emph{dual sequence} by
\[
\prim^*(m)=\frac{m}{\prim(m)}, \quad m\in\NN.
\]
We record for further reference a well-known result.
\begin{lemma}\label{lem:DNI}
Let $\XB=(\xx_n)_{n\in\Nt}$ be a bounded family in a Banach space $\XX$. Assume that $\xx_n\not=0$ for some $n\in\Nt$. Then the dual sequence of the fundamental function of $\XB$ is nondecreasing.
\end{lemma}

\begin{proof}
Although the definition of fundamental function used by the authors of \cite{DKKT2003} differs from ours, their proof can be translated verbatim to the notion used here (see \cite{DKKT2003}*{Comments preceding Lemma 2.2}).
\end{proof}

We say that a complete minimal system $\XB$ of $\XX$ is \emph{$C$-bidemocratic}, $1\le C<\infty$, if
\[
\norm{\Ind_{\varepsilon,A}[\XB,\XX]} \norm{\Ind_{\delta,B}[\XB^*,\XX^*]} \le C \max\enbrace{\abs{A},\abs{B}}
\]
for all $A$, $B\in\Pt_{<\infty}(\Nt)$ and all $\varepsilon\in\EE^A$, $\delta\in\EE^B$. We say that $\XB$ is bidemocratic if it is $C$-bidemocratic for some $C$. If $C=1$, we say that $\XB$ is \emph{isometrically bidemocratic.} If $\XB$ is $C$-bidemocratic, then both $\XB$ and $\XB^*$ are $C$-superdemocratic. Hence, any bidemocratic complete minimal system is a basis.

Bidemocracy can be characterized using the fundamental function. In fact, $\XB$ is $C$-bidemocratic if and only if
\begin{equation*}
\usdf[\XB,\XX](m) \, \usdf[\XB^*,\XX^*](m) \le Cm,\quad m\in\NN.
\end{equation*}
Conversely, any complete minimal system satisfies the inequality
\begin{equation}\label{eq:SemiBi}
m \le \usdf[\XB,\XX](m) \, \lsdf[\XB^*,\XX^*](m) ,\quad m\in\NN.
\end{equation}

It is known that subsymmetric bases are bidemocratic \cite{LinTza1977}*{Proposition 3.a.6}.

A finite set $A\subset\Nt$ is said to be a \emph{greedy set} of $f\in\XX$ relative to a basis $\XB$ with coordinate functionals $\XB^*=(\xx_n^*)_{n\in\Nt}$ if
\[
\abs{\xx_n^*(f)} \ge \abs{\xx_k^*(f)}, \quad n\in A, \, k\in\Nt \setminus A.
\]
We denote by $\Gt(f)=\Gt[\XB,\XX](f)$ the set all all greedy sets of $f$. Note that $\emptyset\in\Gt(f)$. Since
\[
\Ft(f)=\Ft[\XB,\XX](f):=(\xx_n^*(f))_{n=1}^\infty\in c_0(\Nt),
\]
for any $A$, $B\in\Gt(f)$ with $A\subset B$ and any $m\in\NN$ with $\abs{A}\le m \le \abs{B}$, there is $G\in\Gt(f)$ with $A\subset G \subset B$ and $\abs{G}=m$.

We will denote by $\Rt[\XB,\XX](f)$ the nonincreasing rearrangement of $\abs{\Ft(f)}$.

Using our notation, a basis $\XB$ is $C$-greedy, $C\in[1,\infty)$, if
\begin{equation}\label{eq:GC}
\norm{ f-S_A(f)} \le C \norm{ f-\sum_{n\in B} a_n\,\xx_n}
\end{equation}
for all $f\in\XX$, all $A\in\Gt(f)$, all $B\in \Pt(\Nt)$ with $\abs{B}\le\abs{A}$, and all $(a_n)_{n\in B}$ in $\FF$. We will denote by $K_g[\XB]$ the smallest constant $C$.

If \eqref{eq:GC} holds in the particular case that
\[
a_n=\xx_n^*(f), \quad n\in B,
\]
$\XB$ is \emph{$C$-almost greedy}. Note that it suffices to impose this condition in the case when $A\cap B=\emptyset$. We say that $\XB$ is a \emph{greedy basis} (resp., \emph{almost greedy basis}) if it is $C$-greedy (resp., \emph{$C$-almost greedy basis}) for some $C$. If $C=1$, we say that $\XB$ is \emph{isometrically greedy} (resp., \emph{isometrically almost greedy}).

Konyagin and Telmakov \cite{KoTe1999} proved that a basis $\XB$ is greedy if and only if it is unconditional and democratic. Under the assumption of unconditionality, democracy, super-democracy and symmetry for largest coefficients are equivalent properties. In this regard, we point out that symmetry for largest coefficients plays an important role when seeking for optimal constants for greediness. Indeed, if $\XB$ is $C$-greedy, then it is suppression $C$-unconditional and $C$-symmetric for largest coefficients. Conversely, if $\XB$ is suppression $C_1$-unconditional and $C_2$-symmetric for largest coefficients, then it is $C_1 C_2$-greedy (see \cite{AlbiacAnsorena2017b}*{Remark 2.6}). We have the following.

\begin{theorem}[\cite{AW2006}*{Theorem 3.4}]\label{thm:AW}
Let $\XB$ be a basis of a Banach space $\XX$. Then $\XB$ is isometrically greedy if and only if it is suppression $1$-unconditional and has Property~(A).
\end{theorem}

The study of \emph{quasi-greedy} and \emph{almost greedy} bases has
provided crucial technical tools for renorming arguments. Recall that a basis $\XB$ is $C$-quasi-greedy (\cite{KoTe1999}), $1\le C<\infty$, if
\[
\norm{S_A(f)} \le C \norm{f}, \quad f\in\XX, \, A\in\Gt(f),
\]
and we say $\XB$ is a \emph{quasi-greedy basis} if it is $C$-quasi-greedy for some $C$ . If $\XB$ is quasi-greedy, then there is a (possiby different) constant $C\in[1,\infty)$ such that
\[
\norm{f-S_A(f)} \le C \norm{f}, \quad f\in\XX, \, A\in\Gt(f).
\]
If this inequality holds with a particular constant $C$, then we say that $\XB$ is suppression $C$-quasi-greedy. If $\XB$ is quasi-greedy, then democracy, super-democracy and symmetry for largest coefficients are still equivalent properties. It is known \cite{DKKT2003}*{Theorem 3.3} that a basis is almost greedy if and only if it quasi-greedy and democratic. Quantitatively, if $\XB$ is suppression $C_1$-quasi-greedy and $C_2$-symmetric for largest coefficients, then it is almost $C_1 C_2$-greedy (see \cite{AlbiacAnsorena2017b}*{Theorem 3.3}). Besides, in the isometric case the quasi-greedy condition can be dropped:

\begin{theorem}[\cite{AlbiacAnsorena2017b}*{Theorem 2.3}]\label{thm:1AG}
Let $\XB$ be a basis of a Banach space $\XX$. Then $\XB$ is isometrically almost greedy if and only if it has Property~(A).
\end{theorem}

Several important results on quasi-greedy bases underpin our approach:

\begin{itemize}[leftmargin=*]
\item Quasi-greedy bases satisfy Lorentz-type embeddings of the form
\begin{equation}\label{eq:TCG}
\lsdf[\XB,\XX](m) \Rt[\XB,\XX](f) (m) \le C^2 \norm{f}, \quad m\in\NN,
\end{equation}
relating the rearrangement of coefficients to the norm (see \cite{AABW2021}*{Chapter 9}, cf.\@ \cite{DKKT2003}*{Lemma 2.2}).
\item Quasi-greedy bases have slow-growing unconditionality parameters:
\[
\unc[\XB,\XX](m)\lesssim \log(1+m), \quad m\in\NN,
\]
(see \cite{DKK2003}*{Lemma 8.2}). Besides, if $\XX$ is superreflexive, then this estimate can be improved. Namely,
\[
\unc[\XB,\XX](m)\lesssim \log^\alpha(1+m), \quad m\in\NN,
\]
for some $0<\alpha<1$ (see \cite{AAGHR2015}*{Theorem 1.1}).
\end{itemize}

A key role in our work is also played by regularity properties of the fundamental function. A sequence $\prim\colon\NN\to(0,\infty)$ is said to have the \emph{lower regularity property} (LRP for short) if there is $r\in\NN$ such that
\[
\prim(rm)\ge 2 \prim(m), \quad m\in\NN.
\]
In turn, $\prim$ is said to have the \emph{upper regularity property} (URP for short) if
its \emph{dual sequence} has the LRP, i.e., there is $r\in\NN$ such that
\[
\prim(rn)\le \frac{1}{2} r \prim(m), \quad m\in\NN.
\]

Regularity of the fundamental function is a geometric property reflecting probabilistic features of the space, and it is critically involved in scaling arguments in greedy-type renormings. LRP and URP arise naturally in spaces with nontrivial type or cotype via the following results of Dilworth et al.\@ \cite{DKKT2003}.

\begin{lemma}[\cite{DKKT2003}*{Proposition 4.4}]\label{lem:URPBD}
Let $\XB$ be an almost greedy basis of Banach space $\XX$. If the fundamental function of $\XB$ has the URP, then $\XB$ is bidemocratic.
\end{lemma}

\begin{lemma}[\cite{DKKT2003}*{Proposition 4.1(1)}]\label{lem:CTLRP}
Let $\XB$ be a quasi-greedy basis of a Banach space $\XX$ with nontrivial cotype. Then $\lsdf[\XB,\XX]$ has the LRP.
\end{lemma}

\begin{lemma}[\cite{DKKT2003}*{Proposition 4.1(2)}]\label{lem:TURP}
Let $\XB$ be a quasi-greedy basis of Banach space $\XX$ with nontrivial type. Then $\usdf[\XB,\XX]$ has the URP.
\end{lemma}

Linear properties of bases, such as unconditionality and subymmetry, trivially pass to dual bases. Since quasi-greediness is of nonlinear nature, we cannot rely on dual bases to inherit this property in general. However, quasi-greediness passes in a controlled way to dual bases and some interesting partial results have been achieved in this regard. We will use the following one.

\begin{theorem}[\cite{DKKT2003}*{Theorem 5.4}]\label{thm:DKKTDual}
Let $\XB$ be a bidemocratic quasi-greedy basis of a Banach space $\XX$. Then $\XB^*$ is a quasi-greedy basis of its closed linear span in $\XX^*$.
\end{theorem}

We will use the fact that if $\XB$ is quasi-greedy, then the above-mentioned canonical map from $\XX$ to $\enbrak{\XB^*}^*$ is an isomorphic embedding. This was first proved for bidemocratic bases \cite{AABW2021}*{Theorem 10.15} and later on extended to general quasi-greedy bases in \cite{Berasategui2025}.

\begin{proposition}[\cite{Berasategui2025}*{Theorem 3.4}]\label{prop:reflex}
Let $C\in[1,\infty)$ and $\XB$ be a $C$-quasi-greedy basis of a Banach space $\XX$. Then, for any $f\in\XX$,
\[
\norm{f}\le \frac{1}{C}\sup\enbrace{ \abs{f^*(f)} \colon f\in \enbrak{\XB^*}, \, \norm{f^*}\le 1}.
\]
\end{proposition}
\section{A unified renorming scheme for greedy-type bases}\label{sect:construction}\noindent
The purpose of this section is to develop the analytic machinery that underlies all renorming results proved later in the paper. Although our main theorems concern greedy and almost greedy bases, the construction is carried out in a more general framework that includes quasi-greedy and bidemocratic systems as well. The central challenge is to build an equivalent norm that simultaneously controls the contributions of the large coefficients of a vector, its small coefficients, and the action of the dual basis, while preserving the fundamental function up to a prescribed regularization. To achieve this balance, we combine Marcinkiewicz-type expressions, duality estimates arising from quasi-greediness, and a careful use of Dini-regularizations of fundamental functions.

The norms introduced here are designed so that the symmetry and regularity properties required for isometric forms of greediness are encoded directly into their structure. Once these renormings are in place, the proofs of the main results become largely transparent, relying only on the structural consequences established in this section.

We say that $\prim$ is \emph{Dini-regular} if its \emph{discrete derivative} is equivalent to $1/\prim^*$, that is there is a constant $C$ such that
\[
\frac{1}{C}\frac{\prim(m)}{m} \le \prim(m)-\prim(m-1) \le C \frac{\prim(m)}{m}, \quad m\in\NN.
\]
Here, and throughout the paper, we use the convention that sequences of positive numbers are extended to the set of nonnegative integers by $0\mapsto 0$.

While lower regularity and upper regularity pass to equivalent sequences, Dini regularity does not. So, given sequences $\prim$ and $\primt$ of of positive numbers, we say that $\prim$ is a \emph{Dini regularization} of $\primt$ if it is Dini-regular and equivalent to $\primt$.

\begin{lemma}\label{lem:LRPDini}
Let $\primt$ be a sequence of positive numbers. Assume that $\primt$ and $\primt^*$ are nondecreasing. Then $\primt$ has the LRP if and only if it admits a Dini regularization $\prim$. Besides, we can choose $\prim$ so that $\prim$ and $\prim^*$ are nondecreasing.
\end{lemma}

\begin{proof}
Assume that $\primt$ has the LRP. Define
\[
\prim(m)=\sum_{n=1}^m \frac{\primt(n)}{n}, \quad m\in\NN.
\]
It is clear that $\prim$ and $\prim^*$ are nondecreasing, and that $\primt\le \prim$. By \cite{AlbiacAnsorena2016}*{Lemma 2.12}, there is a constant $C$ such that $\prim \le C \primt$ for all $m\in\NN$. Finally, for all $m\in\NN$,
\[
\prim^*(m) \enpar{\prim(m)-\prim(m-1)}=\frac{\primt(m)}{\prim(m)} \in \enbrak{\frac{1}{C} ,1}.
\]

Assume that $\prim$ is a Dini regularization of $\primt$. Then,
\[
\sum_{n=1}^m \frac{\primt(n)}{n}\approx \sum_{n=1}^m \frac{\prim(n)}{n} \approx \sum_{n=1}^m \prim(n)-\prim(n-1)=\prim(m) \approx \primt(m)
\]
for $m\in\NN$. By \cite{AlbiacAnsorena2016}*{Lemma 2.12}, $\primt$ has the LRP.
\end{proof}

Given a complete minimal system $\XB=(\xx_n)_{n\in\Nt}$ of a Banach space $\XX$ with dual basis $(\xx_n^*)_{n\in\Nt}$ we define, for $f\in\XX$, $f^*\in\XX^*$ and $A\in\Pt_{<\infty}(\Nt)$,
\begin{align*}
\SNA[\XB,\XX](f,A)&=\sum_{n\in A} \abs{\xx_n^*(f)}, \\
\SNC[\XB,\XX](f,f^*,A)&=S_A(f^*)(f)=\sum_{n\in A} f^*(\xx_n) \xx_n^*(f), \\
\SNB[\XB,\XX](f,f^*,A)&=\enpar{f^*-S_A(f^*)}(f)=f^*(f)-\SNC(f,f^*,A).
\end{align*}
If the basis and the space are clear from context, we will simply put $\SNA(f,A)$, $\SNC(f,B)$ and $\SNB(f,C)$. Loosely speaking,
\[
\SNB(f,f^*,A)=\sum_{n\in \Nt\setminus A}f^*(\xx_n) \xx_n^*(f).
\]
As a matter of fact, if $\XB$ is quasi-greedy, then this series converges when arranged according to size of $\abs{\xx_n^*(f)}$ (see \cite{Woj2000}*{Theorem 1} and \cite{AABW2021}*{Theorem 4.1}).

We write down two elementary results that we will use several times.
\begin{lemma}\label{lem:Marcin}
Let $\XB=(\xx_n)_{n\in\Nt}$ be a complete minimal system of a Banach space $\XX$. Then for any $f\in\XX$ and $A\in\Pt_{<\infty}(\Nt)$,
\[
\SNA(f,A)\le \usdf[\XB^*,\XX^*] \enpar{\abs{A}} \norm{f}.
\]
\end{lemma}
\begin{proof}
If $\varepsilon=(\varepsilon_n)_{n\in A}$ in $\EE$ such that $\xx_n^*(f)=\varepsilon_n \abs{\xx_n^*(f)}$ for all $n\in A$,
\begin{align*}
\SNA(f,A)
&=\Ind_{\overline{\varepsilon},A}[\XB^*,\XX^*](f)\\
&\le \norm{\Ind_{\overline{\varepsilon},A}[\XB^*,\XX^*]} \norm{f}
\le \usdf[\XB^*,\XX^*](\abs{A}) \norm{f}.\qedhere
\end{align*}
\end{proof}

\begin{lemma}\label{lem:DualBound}
Let $\XB=(\xx_n)_{n\in\Nt}$ be a complete minimal system of a Banach space $\XX$ and $\prim\colon\NN\to(0,\infty)$ be a nondecreasing sequence. Assume that
\[
\SNA(f,A)\le \prim(\abs{A})\norm{f}, \quad f\in\XX, \, A\in\Pt_{<\infty}(\Nt).
\]
Then $\usdf[\XB^*,\XX^*]\le \prim$.
\end{lemma}
\begin{proof}
Given $m\in\NN$, $A\in\Pt_{\le m}(\Nt)$ and $\varepsilon\in\EE^A$,
\[
=\sup\enbrace{\abs{\Ind_{\varepsilon,A}[\XB^*,\XX^*](f)} \colon f\in B_\XX} \le \SNA(f,A) \le \prim\enpar{\abs{A}}\le \prim(m).\qedhere
\]
\end{proof}

\begin{theorem}\label{thm:main}
Let $\XB=(\xx_n)_{n\in\Nt}$ be a quasi-greedy bidemocratic basis of a Banach $\XX$. If $\usdf[\XB,\XX]$ has the LRP, then $\XX$ can be renormed so that $\XB$ isometrically bidemocratic and has Property~(A). Besides, there is a Dini regularization $\prim$ of $\usdf[\XB,\XX]$ such that $\prim$ and $\prim^*$ are nondecreasing. Given such a sequence $\prim$, we can choose the renorming so that the fundamental function of $\XB$ relative to it is $\prim$.
\end{theorem}

\begin{proof}
Let $\XB^*=(\xx_n^*)_{n\in\Nt}$ be the dual system of $\XB$, and put $\YY=[\XB^*]$.
Let $C_d$ be such that $\XB$ is $C_d$-bidemocractic. By Theorem~\ref{thm:DKKTDual}, there is $C_q$ in $[1,\infty)$ such that $\XB^*$ is a suppression $C_q$-quasi-greedy basis of $\YY$. By Lemma~\ref{lem:DNI} and Lemma~\ref{lem:LRPDini}, a sequence $\prim$ as in the statement exists. By \eqref{eq:SemiBi} and \eqref{eq:TCG}, there is a constant $C_e$ such that
\begin{equation}\label{eq:LorentzE}
\sup_{m\in\NN} \prim^*(m) \Rt[\XB^*,\YY](f^*)(m) \le C_e \norm{f^*}, \quad f^*\in\YY.
\end{equation}
Let $C_a$, and $C_r$ in $(0,\infty)$ be such that
\[
\prim(m) \le C_a\usdf[\XB,\XX](m), \quad \prim^*(m)\enpar{\prim(m)-\prim(m-1)}\ge \frac{1}{C_r}, \quad m\in\NN.
\]
We have $\abs{\SNB(f,f^*,A)} \le C_q \norm{f^*} \norm{f}$ and, by Lemma~\ref{lem:Marcin},
\[
\frac{\SNA(f,A)}{\prim^*\enpar{\abs{A}}} \le C_a C_d \norm{f},
\]
with the convention that $0/0=0$.

Consequently, taking into account Proposition~\ref{prop:reflex}, if we set
\[
\Omega=\enbrace{(A,B,f^*) \in \enpar{\Pt(\Nt)}^2\times B_\YY \colon B\in\Gt(f^*),\, \abs{A} \le \abs{B} },
\]
and we define $\Nnorm{\cdot}\colon \XX\to[0,\infty]$ by
\[
\Nnorm{f}=
\sup\enbrace{ \frac{ \SNA(f,A)}{\prim^*\enpar{\abs{A}}} + \frac{\abs{\SNB(f,f^*, A\cup B)}}{2C_e C_r} \colon (A,B,f^*)\in\Omega},
\]
then $\Nnorm{\cdot}$ is an equivalent norm for $\XX$. We will infer the properties of this renorming from the following assertion.
\begin{claim}\label{clm:A}
If $f\in\XX$ and $D\in\Gt(f)$ are such that $\abs{\xx_n^*(f)}=1$ for all $n\in D$, and we put $T=\supp(f)$, then
\[
\Nnorm{f}=
\sup\enbrace{ \frac{\SNA(f,A)}{\prim^*\enpar{\abs{A}}} +\frac{\abs{\SNB(f,f^*, A\cup B)}}{2C_e C_r}\colon (A,B,f^*)\in\Omega, D \subset A\subset T}.
\]
\end{claim}

Indeed, Claim~\ref{clm:A} implies that, given $A\in\Pt_{<\infty}(\Nt)$, $\varepsilon\in\EE^A$, and $g\in\XX$ with $S:=\supp(g)\subset\Nt\setminus A$ and $\abs{\xx_n^*(g)}\le 1$ for all $n\in\Nt$,
\begin{align*}
&\Nnorm{\Ind_{\varepsilon,A}[\XB,\XX]+g}=N(g,A)\\
&:=\sup\enbrace{ \frac{\abs{A}+\SNA(g,E)}{\prim^*\enpar{\abs{A}+\abs{E}}} +
\frac{\abs{\SNB(g,g^*,E \cup B)}}{2C_e C_r} \colon (E,B,g^*)\in\Upsilon(g,A)},
\end{align*}
where
\[
\Upsilon(g,A)=\enbrace{(E,B,g) \in\Pt(S) \times \Pt(\Nt) \times B_\YY \colon B\in\Gt(g^*), \, \abs{A}+\abs{E} \le \abs{B} }.
\]
For a fixed $g\in\XX$, $N(g,A)$ only depends on $A$ through $\abs{A}$. Hence, $\XB$ has Property~(A) relative to $\Nnorm{\cdot}$. Besides,
\[
\Nnorm{\Ind_{\varepsilon,A}[\XB,\XX]}=N(0,A)= \frac{\abs{A}}{\prim^*{\enpar{\abs{A}}}}=\prim\enpar{\abs{A}}.
\]
Since $\prim$ is nondecreasing, the fundamental function of $\XB$ relative to the renorming is $\prim$. In turn, that of $\XB^*$ is bounded by $\prim^*$ by Lemma~\ref{lem:DualBound}. Hence, $\XB$ is isometrically bidemocratic relative to $\Nnorm{\cdot}$.

We conclude by proving Claim~\ref{clm:A}. To that end, we pick $f^*\in S_\YY$, $B\in\Gt(f^*)$ and $A\subset\Nt$ with $\abs{A} \le \abs{B}$. If $D\not\subset A$, we pick $j\in D\setminus A$. Set $A'=A\cup \{j\}$. Choose $B'\in \Gt(f^*)$ with $B\subset B'$ and $\abs{B'}=\abs{B}+1$. Put $m=\abs{A}$, $t= \sum_{n\in A} \abs{\xx_n^*(f)}$ and $E= A'\cup B' \setminus(A\cup B)$. We have
\begin{align*}
&\rho:=\frac{\SNA(f,A)}{\prim^*\enpar{\abs{A}}} +\frac{\abs{\SNB(f,f^*, A\cup B)}}{2 C_e C_r}
-\frac{\SNA(f,A')}{\prim^*\enpar{\abs{A'}}} -\frac{\abs{\SNB(f,f^*, A'\cup B'})}{2 C_e C_r}\\
&\le\frac{1}{2C_e C_r} \sum_{n\in E} \abs{f^*(\xx_n)} \abs{\xx_n^*(f)}+ \enpar{\frac{1}{\prim^*(m)}- \frac{1}{\prim^*(m+1)}} t
-\frac{1}{\prim^*(m+1)}.
\end{align*}
Since $E\subset\Nt \setminus B$ and $\abs{B}\ge m$, by \eqref{eq:LorentzE},
\[
\abs{f^*(\xx_n)} \le C_e \frac{1}{\prim^*(m+1)}
\]
for all $n\in E$. Since $\abs{E} \le 2$ and $\abs{\xx_n^*(f)} \le 1$ for all $n\in\Nt$,
\begin{align*}
\rho
&\le \frac{1}{C_r} \frac{1}{\prim^*(m+1)} + \enpar{\frac{1}{\prim^*(m)}- \frac{1}{\prim^*(m+1)}} m -\frac{1}{\prim^*(m+1)} \\
&= \frac{1}{C_r} \frac{1}{\prim^*(m+1)} -\enpar{\prim(m+1)-\prim(m)}\le 0.
\end{align*}
By induction on $\abs{D\setminus A}$ we infer that the suppremum that defines $\Nnorm{f}$ is approached through triples $(A,B,f^*)$ with $D\subset A$. Since
\[
\SNA(f,A)=\SNA(f,A\cap T), \quad \SNB(f,f^*, A \cup B)=\SNB(f,f^*,(A\cap T)\cup B),
\]
this suppremum is approached through triples $(A,B,f^*)\in\Omega$ with $D \subset A \subset T$.
\end{proof}

We do not know whether the assumption that the fundamental function has the LRP can be dropped from the statement of Theorem~\ref{thm:main}. Notwithstanding, in this situation, we can still achieve quasi-isometric renormings. We will next prove this result, which is a quasi-greedy counterpart of \cite{DKOSZ2014}*{Theorem A}. We point that the method used by the authors of \cite{DKOSZ2014} to obtain quasi-isometric renormings relative to greediness for bidemocratic unconditional bases still yields quasi-isometric renormings relative to symmetry for largest coefficients when unconditionality is lifted. Since we do not have to care about the constants related to unconditionality, these renormings plainly yield quasi-isometric renormings relative to greediness. However, in the quasi-gredy case we also need to pay attention to the constants related to suppression quasi-greediness. As matter of fact, the renorming we will use significantly differs from the renormings we employed in the proofs of Theorem~\ref{thm:main} and \cite{DKOSZ2014}*{Theorem A}.

\begin{theorem}
Let $\XB=(\xx_n)_{n\in\Nt}$ be a bidemocratic quasi-greedy basis of a Banach space $\XX$. Then for any $\varepsilon>0$, $\XX$ can be renormed so that $\XB$ is isometrically bidemocratic and almost $(1+\varepsilon)$-greedy. Moreover, if $\prim\approx\usdf[\XB,\XX]$ is nondecreasing, and $\prim^*$ is also nondecreasing, then we can get that the fundamental function of $\XB$ relative to the renorming is $\prim$.
\end{theorem}

\begin{proof}
Let $C_a$, $C_d$, $C_q$, $C_e$, $\XB^*$ and $\YY$ be as in the proof of Theorem~\ref{thm:main}. Also, let $C_b\in(0,\infty)$ be such that $\usdf[\XB,\XX]\le C_b \prim$. Set
\[
\Delta=\enbrace{ (A,B,f^*)\in \enpar{\Pt(\Nt)}^2\times B_\YY \colon A,B\in\Gt(f^*), \, B\subset A}.
\]
Fix $\delta>0$. Define for $f\in\XX$
\begin{align*}
\norm{f}_s&=\sup\enbrace{ \frac{\SNA(f,A)}{\prim^*\enpar{\abs{A}}} \colon A\in\Pt_{<\infty}(\Nt)}, \\
\norm{f}_t&= \sup\enbrace{ \abs{\SNC(f,f^*,A\setminus B)} \colon (A,B,f^*)\in\Delta},\\
\Nnorm{f}&=\max\enbrace{ \norm{f}_s , \delta \norm{f}_t}.
\end{align*}
Note that $\norm{f}_s$ is a Marcinkiewicz norm evaluated at $\Ft(f)$. We have $\norm{f}_t\le 2C_q \norm{f}$ and, by Lemma~\ref{lem:Marcin}, $\norm{f}_s\le C_a C_d \norm{f}$. Hence, by Proposition~\ref{prop:reflex}, $\Nnorm{\cdot}$ is an equivalent norm for $\XX$.

Let $D\in\Gt(f)$ and $E\in\Pt(\Nt)$ be such that $E\cap D=\emptyset$ and $\abs{E} \le m:= \abs{D}$. The $1$-symmetry of the Marcinkiewicz norm yields
\begin{equation}\label{eq:SymMar}
\rho:=\norm{f-S_D(f)}_s \le \norm{f-S_E(f)}_s.
\end{equation}

Pick $(A,B,f^*)\in\Delta$. If $\abs{B}\ge m$ we set $A_0=B_0=\emptyset$, $A_1=A$ and $B_1=B$. If $\abs{A}\le m$, we set pick $B_1\in\Gt(f^*)$ with $\abs{B_1}=m$, and we set $A_0=A$, $B_0=B$ and $A_1=B_1$. Finally, if $ \abs{B} <m <\abs{A}$, we pick $A_1\in\Gt(f^*)$ with $B\subset A_1 \subset A$ and $\abs{A_1}=m$, and we set $A_0=B$, $B_0=A_1$ and $B_1=A$. In any case we have $(A_1,B_1,f^*)\in\Delta$, $\abs{B_1}\ge m$, $\abs{A_0}\le m$ and
\[
\SNC(f-S_D(f),f^*,A\setminus B)=\rho_0+\rho_1,
\]
where
\[
\rho_0=\SNC(f-S_D(f),f^*,B_0\setminus A_0), \quad \rho_1=\SNC(f-S_D(f),f^*,B_1\setminus A_1).
\]
We also have $\rho_1=\rho_2+\rho_3$, where
\begin{align*}
\rho_2=& \SNC(f-S_E(f),f^*,B_1\setminus A_1);\\
\rho_3=& \SNC(f,f^*,E\cap B_1\setminus A_1)-\SNC(f,f^*,D\cap B_1\setminus A_1)
\end{align*}
On the one hand, by Lemma~\ref{lem:Marcin},
\begin{align*}
\rho_0
&\le \sup_{n\in\Nt\setminus D} \abs{\xx_n^*(f)} \sum_{n\in A_0} \abs{f^*(\xx_n)} \\
&\le \min_{n\in D} \abs{\xx_n^*(f)} \sum_{n\in A_0} \abs{f^*(\xx_n)} \\
&= \min_{n\in D} \abs{\xx_n^*(f-S_E(f))} \sum_{n\in A_0} \abs{f^*(\xx_n)} \\
&\le \frac{1}{\abs{D}} \SNA(f-S_E(f),D) \sum_{n\in A_0} \abs{f^*(\xx_n)}\\
&\le \frac{\prim^*(m)}{m} \norm{f-S_E(f)}_s \sum_{n\in A_0} \abs{f^*(\xx_n)}\\
&\le \frac{\norm{f-S_E(f)}}{\prim(m)} \usdf[\XB,\XX]\enpar{\abs{A_0}} \le C_b\norm{f-S_E(f)}.
\end{align*}
On the other hand,
\begin{align*}
\rho_3&\le \SNA(f, D\cup E) \sup_{n\in\Nt \setminus B_1} \abs{f^*(\xx_n)}
\le \frac{C_e}{\prim^*(m+1)}\SNA(f, D\cup E) \\
&\le \frac{2 C_e}{\prim^*(m+1)} \SNA(f, D)
=\frac{2 C_e}{\prim^*(m+1)} \SNA(f-S_E(f), D)\\
&\le \frac{2 C_e}{\prim^*(m+1)} \prim^*(m) \norm{f-S_E(f)}_s
\le 2C_e \norm{f-S_E(f)}_s.
\end{align*}
Summing up,
\begin{equation}\label{eq:MiguelEst}
\abs{\rho_0+\rho_1-\rho_2}\le (2C_e+C_b)\norm{f-S_E(f)}.
\end{equation}

Set $D= (2C_e+C_b)\delta$. Combining \eqref{eq:SymMar} with \eqref{eq:MiguelEst} yields
\begin{align*}
\Nnorm{f-S_D(f)}
&\le \max\enbrace{ \norm{f-S_E(f)}_s, \delta \norm{f-S_E(f)}_t + D \norm{f-S_E(f)}_s}\\
&\le\enpar{1+D} \Nnorm{f-S_E(f)}.
\end{align*}

Pick $A\in\Pt_{<\infty}(\Nt)$ and $\varepsilon\in\EE^A$. We have
\[
\norm{\Ind_{\varepsilon,A}[\XB,\XX]}_t\le 2 C_b C_q \prim\enpar{\abs{A}}
\]
and, since $\prim^*$ is nondecreasing,
\[
\norm{\Ind_{\varepsilon,A}[\XB,\XX]}_s=\prim\enpar{\abs{A}}.
\]
Thus, if $D \delta \le 1$ we have $\Nnorm{\Ind_{\varepsilon,A}[\XB,\XX]}=\prim\enpar{\abs{A}}$, and if this is the case, the fundamental function of $\XB$ relative to $\Nnorm{\cdot}$ is $\prim$. By Lemma~\ref{lem:DualBound}, $\XB$ is isometrically bidemocratic relative to the renorming. Choosing $\delta$ small enough so that $D\le \min\{\varepsilon,1\}$, we are done.
\end{proof}
\section{Renorming spaces with greedy bases}\label{Sect:Renorm}\noindent
In this section we use the renorming scheme developed in the preceding section in the specific case of greedy bases. The results obtained here show that once the structural hypotheses are in place (most notably unconditionality, bidemocracy, and the regularity properties of the fundamental function) the abstract methods of Section~\ref{sect:construction} yield isometric versions of the classical greedy inequalities. In particular, the renormings constructed earlier can now be combined with standard renormings to ensure lattice $1$-unconditionality and exact control of the fundamental function. This allows us to recover and strengthen many known renorming theorems for greedy bases and, more importantly, to deduce new ones under mild assumptions on the ambient space. The main outcome of this section is an isometric greedy renorming result (Theorem~4.3), which serves as the foundational step for all subsequent applications, including the resolution of the Haar system renorming problem in Section~\ref{Sec:Haar}.

\begin{lemma}\label{lem:UncRen}
Let $\XB$ be a $C$-greedy basis of a Banach space $\XX$, $1\le C<\infty$. Then there is a renorming of $\XX$ so that $\XB$ is lattice $1$-unconditional and $C$-greedy relative to the renorming. If $\psi$ and $\psi^*$ are the fundamental functions of $\XB$ and $\XB^*$, respectively, relative to the renorming, then we can choose the renorming so that $\psi= \usdf[\XB,\XX]$ and $\psi^*\le \usdf[\XB^*,\XX^*]$.
\end{lemma}

\begin{proof}
The obvious renorming
\[
f\mapsto\sup\enbrace{\norm{M_\lambda(f)}\colon\lambda\in B_{\ell_\infty}}
\]
does the job.
\end{proof}

\begin{theorem}\label{thm:greedy}
Let $\XB=(\xx_n)_{n\in\Nt}$ be a unconditional bidemocratic basis of a Banach space $\XX$. If $\usdf[\XB,\XX]$ has the LRP, then $\XX$ can be renormed so that $\XB$ isometrically bidemocratic and isometrically greedy. Besides, there is a Dini regularization $\prim$ of $\usdf[\XB,\XX]$ such that $\prim$ and $\prim^*$ are nondecreasing. Given such a sequence $\prim$, we can choose the renorming so that the fundamental function of $\XB$ relative to it is $\prim$.
\end{theorem}

\begin{proof}
Just combine Lemma~\ref{lem:UncRen}, Theorem~\ref{thm:main} and Theorem~\ref{thm:AW}.
\end{proof}

Theorem~\ref{thm:greedy} forms a natural bridge between the structural conditions developed earlier and the full isometric greedy theory. Its strength comes from the fact that unconditionality, bidemocracy, and the LRP together allow one to transfer the renorming produced in Section~\ref{sect:construction} directly to the setting of classical greedy inequalities. To make explicit how broadly this mechanism applies, we highlight below several situations in which the hypotheses of Theorem~\ref{thm:greedy} arise from familiar geometric or probabilistic properties of the ambient space. These consequences are standard to experts, but recording them here clarifies the scope of the renorming method and prepares the ground for the main isometric result of this section.

\begin{theorem}\label{thm:LRPURP}
Let $\XB$ be a greedy basis of a Banach space $\XX$. Assume that the fundamental function of $\XB$ has the LRP and the URP. Then $\XX$ can be renormed so that $\XB$ is isometrically greedy, isometrically bidemocratic, and lattice $1$-unconditional. Besides, there is a Dini regularization $\prim$ of $\usdf[\XB,\XX]$ such that $\prim$ and $\prim^*$ are nondecreasing. Given such a sequence $\prim$, we can choose the renorming so that the fundamental function of $\XB$ relative to it is $\prim$.
\end{theorem}

\begin{proof}
Just combine Lemma~\ref{lem:URPBD} with Theorem~\ref{thm:greedy}.
\end{proof}

We close the theoretical part of this section by discussing the role of Rademacher type and cotype.
\begin{theorem}
Let $\XB$ be a bidemocratic unconditional basis of a Banach space $\XX$ with nontrivial cotype. Then $\XX$ can be renormed so that $\XB$ is isometrically greedy, isometrically bidemocratic, and lattice $1$-unconditional. Moreover, there is a Dini regularization $\prim$ of $\usdf[\XB,\XX]$ such that $\prim$ and $\prim^*$ are nondecreasing. Given such a sequence $\prim$ we can choose the renorming so that the fundamental function of $\XB$ relative to it is $\prim$.
\end{theorem}

\begin{proof}
Just combine Lemma~\ref{lem:CTLRP} with Theorem~\ref{thm:greedy}.
\end{proof}

\begin{theorem}
Let $\XB$ be a greedy basis of a Banach space $\XX$ with nontrivial type. Then $\XX$ can be renormed so that $\XB$ is isometrically greedy, isometrically bidemocratic, and lattice $1$-unconditional. Besides, there is a Dini regularization $\prim$ of $\usdf[\XB,\XX]$ such that $\prim$ and $\prim^*$ are nondecreasing. Given such a sequence $\prim$ we can choose the renorming so that the fundamental function of $\XB$ relative to it is $\prim$.
\end{theorem}

\begin{proof}
Just combine Lemma~\ref{lem:TURP}, Lemma~\ref{lem:CTLRP} and Theorem~\ref{thm:LRPURP}.
\end{proof}

In the general case, we can make the most of Theorem~\ref{thm:greedy} as far as applications is concerned when the fundamental function is a power function.

Given $\alpha\in\RR$, we denote by $\prim_\alpha$ the sequence given by $m\mapsto m^\alpha$.

\begin{theorem}\label{cor:potential}
Let $\XB$ be a greedy basis of a Banach space $\XX$. Assume that the fundamental function of $\XB$ is equivalent to $\prim_\alpha$ for some $0<\alpha<1$. Then $\XX$ can be renormed so that $\XB$ is isometrically greedy, isometrically bidemocratic, lattice $1$-unconditional, and its fundamental function is $\prim_a$.
\end{theorem}

\begin{proof}
The power function $\prim_\alpha$ is increasing and Dini-regular for all $\alpha\in(0,1)$, with dual sequence $\prim_{1-\alpha}$. Hence, by Lemma~\ref{lem:LRPDini}, $\usdf[\XB,\XX]$ has the LRP and the URP. So, Theorem~\ref{thm:LRPURP} applies.
\end{proof}
\subsection{The Haar system in \texorpdfstring{$\bm{L_p}$}{}, \texorpdfstring{$\bm{1<p<\infty}$}{}.}\label{Sec:Haar}
The Haar system occupies a central position in classical analysis, acting simultaneously as a model unconditional basis, a canonical wavelet system, and a benchmark example for nonlinear approximation methods. In this section we provide background on the structure and approximation properties of the Haar basis in $L_p:=L_p[0,1]$, $1<p<\infty$, and explain why the problem of achieving isometric greediness through renorming posed in~\cite{AW2006} has been a long-standing problem.

Let $\Dt$ denote the collection of dyadic intervals in $[0,1]$, and for each $I\in\Dt$ let $I_L$ and $I_R$ be its left and right halves. The \emph{Haar function} associated with $I$ is
\[
h_I = \chi_{I_L} - \chi_{I_R}.
\]
Together with the constant function $h_0=1$, these functions form the classical Haar system. Its importance stems not only from its simplicity and orthogonality in $L_2$ but also from the fact that it is the prototypical example of a wavelet system with compact support and minimal smoothness. It is often convenient to normalize the Haar system according the space we study it. Set $\Dt_0=\Dt\cup\{0\}$ and, given $1\le p\le \infty$,
\[
\Ht_p = \enpar{\frac{h_I}{\norm{h_I}_p}}_{I\in\Dt_0}.
\]
Let $p'$ be the conjugate exponent of $p$ defined by
\[
\frac{1}{p}+\frac{1}{p'}=1.
\]
The Haar system is an Schauder basis of $L_p$ for all $p\in[1,\infty)$, and it is unconditional if and only if $p>1$ \cite{Paley1932}. If we identify $L_{p'}$ with the dual space of $L_p$, then dual basis of $\Ht_p$ is $\Ht_{p'}$.

A long-standing question in the geometry of $L_p$ was the optimal value of the unconditionality constant of the Haar basis. In a celebrated series of results, Burkholder established (see \cite{Burk1988}) that $K_u[\Ht_p] = p^* -1$, where
\[
p^* = \max\{p, p'\}.
\]

The nonlinear approximation properties of the Haar basis were studied extensively in the late 1990s. In particular, Temlyakov proved in \cite{Temlyakov1998} that for every $1<p<\infty$ the normalized Haar system $\Ht_p$ is a greedy basis of $L_p$. In fact, there exists a constant $C>0$ such that $\Ht_p$ is $C\,p^{*}$-democratic for all $1<p<\infty$. Combining this estimate with the general bounds for the greedy constant in terms of democracy and unconditionality yields
\[
p^{*}\lesssim C_g[\Ht_p]\lesssim (p^{*})^{3}.
\]
More refined arguments show that the greedy constant of the Haar system grows linearly with $p^{*}$ as $p$ approaches the endpoints $1$ or $\infty$ \cite{AAB2020}.

These estimates indicate that, when $L_p$ is endowed with its classical norm, the greedy behavior of the Haar system deteriorates as $p$ moves away from the Hilbertian case $p=2$. This phenomenon provides strong motivation for investigating renormings of $L_p$ that prevent the growth of the greedy constant and restore optimal greedy symmetry.

The authors of \cite{AW2006} explicitly asked whether the Haar system in $L_p$ can be made isometrically greedy under an equivalent renorming (\cite{AW2006}*{Problem 6.2}). This question has remained open for two decades and is singled out in the literature as one of the most natural and important renorming problems in greedy approximation theory (see \cite{DKOSZ2014}*{Problem 14} or \cite{AAT2025}*{Problem 2}).

Towards a positive answer, Albiac and Wojtaszczyk showed \cite{AW2006}*{Proposition 4.5} that under an equivalent norm, for any $\epsilon>0$, there is a subsequence of the Haar basis which is $1$-unconditional and $(1+\epsilon)$-democratic, and whose closed linear span is isomorphic to $L_p$. Partial progress was made in \cites{DOSZ2011,DKOSZ2014}. Specifically, the authors used Theorem~\ref{thm:DKOSZ:4} to see that $L_p$ can be renormed so that $\Ht_p$ becomes lattice $1$-unconditional, isometrically bidemocratic, and, for fixed $\epsilon>0$, $(1+\varepsilon)$-greedy (see \cite{DOSZ2011}*{Corollary 2.3} and \cite{DKOSZ2014}*{Corollary 5}).

These results demonstrate that the Haar system is remarkably close to being $1$-greedy under suitable renormings. Nevertheless, the precise constant $1$ remained out of reach, and the existing methods lacked the necessary structure to handle the symmetry conditions required by Property~(A) in the isometric theory.

As an application of our techniques from the preceding sections we obtain that, contrary to what was announced (without a proof) in the \emph{Concentration week on greedy algorithms in Banach spaces and compressed sensing} held on July18--22, 2011 at TexasA\&M University, $L_p$ can in fact be renormed so that the Haar system is isometrically greedy.

\begin{theorem}
Let $1<p<\infty$. Then the Lebesgue space $L_p$ can be renormed so that $\Ht_p$
is isometrically greedy, isometrically bidemocratic, and lattice $1$-unconditional, and its fundamental function is $\prim_{1/p}$.
\end{theorem}

\begin{proof}
$\Ht_p$ is a greedy basis of $L_p$ with fundamental function equivalent to $\prim_{1/p}$ (see \cite{Temlyakov1998}). Hence the statement follows directly from Theorem~\ref{cor:potential}.
\end{proof}
\subsection{Direct sums and Besov spaces}\label{Sect:Besov}
The renorming theorems obtained in Section~\ref{Sec:Haar} extend far beyond the classical setting of the Haar system, and this section illustrates this breadth through two broad families of spaces where greedy-type bases arise naturally: Besov sequence spaces and mixed-norm infinite direct sums. These spaces provide a flexible framework for multiscale decompositions and vector-valued approximation, and their canonical bases often exhibit the same power-function fundamental behavior as the Haar system. As a consequence, the renorming machinery developed earlier applies with full force, yielding isometrically greedy and bidemocratic bases under mild structural hypotheses. The results below show that the phenomenon observed for $L_p$ is not exceptional: entire classes of function space models admit renormings that achieve the optimal symmetry and greedy properties simultaneously.

\begin{theorem}
Let $1<p<\infty$ and $1\le q\le \infty$. Then the Besov space
\[
B_{p,q}=\enpar{\bigoplus_{n=1}^\infty \ell_q^n}_{\ell_p}
\]
can be renormed so that it has an isometrically greedy, isometrically bidemocratic, and lattice $1$-unconditional basis whose fundamental function is $\prim_{1/p}$.
\end{theorem}

\begin{proof}
It is known \cite{DFOS2011} that $B_{p,q}$ has a greedy basis whose fundamental function equivalent to $\prim_{1/p}$. Hence, Theorem~\ref{cor:potential} applies.
\end{proof}

\begin{theorem}\label{thm:IDS}
Let $1<p<\infty$ and $C\in[1,\infty)$. Let $J$ be a countable set. Suppose that for each $j\in J$, $\XX_j$ is a Banach space with a suppression $C$-unconditional $C$-democratic basis $\XB_j$ whose fundamental function is $C$-equivalent to $\prim_{1/p}$. Then
\[
\XX=\enpar{\bigoplus_{j\in J} \XX_j}_{\ell_p}
\]
can be renormed so that $\XB:=\oplus_{j\in J} \XB_j$ is isometrically greedy, isometrically bidemocratic and lattice $1$-unconditional, and its fundamental function is $\prim_{1/p}$.
\end{theorem}

\begin{proof}
It is known (see \cite{AABW2021}*{Lemma 11.12}) that $\XB$ is an unconditional and democratic basis with fundamental function equivalent to $\prim_{1/p}$. So, Theorem~\ref{cor:potential} applies.
\end{proof}

For finite direct sums we give a result that works for fundamental functions that are not potential.

\begin{theorem}\label{thm:FDS}
Let $\prim$ be a Dini-regular sequence such that $\prim$ and $\prim^*$ are nondecreasing. Let $J$ be a finite set. Suppose that for each $j\in J$, $\XX_j$ is a Banach space with a suppression unconditional bidemocratic basis $\XB_j$ whose fundamental function is equivalent to $\prim$. Then $\XX=\oplus_{j\in J} \XX_j$ can be renormed so that $\XB=\oplus_{j\in J} \XB_j$ is isometrically greedy, isometrically bidemocratic and lattice $1$-unconditional, and its fundamental function is $\prim$.
\end{theorem}

\begin{proof}
By \cite{GHO2013}*{Proposition 6.1}, $\XB$ is an unconditional democratic basis of $\XX$ with fundamental function equivalent to $\prim$. So, Theorem~\ref{thm:greedy} applies.
\end{proof}

\begin{remark}
Theorem~\ref{thm:FDS} (or Theorem~\ref{thm:IDS}) overrides \cite{DOSZ2011}*{Theorem6.9}, which states that $\ell_2\oplus \ell_{2,1}$ can be renormed so that its canonical basis is isometrically greedy. This result from \cite{DOSZ2011} solved a question posed in \cite{AW2006} by providing an isometrically greedy basis that fails to be subsymmetric.
\end{remark}
\section{Renorming spaces with subsymmetric bases}\label{Sect:Subsym}\noindent
In this section we turn to the case of subsymmetric bases, where the interplay between symmetry and greedy-type properties allows for a renorming approach that differs in nature from the arguments used earlier in the paper. In contrast to the unconditional setting of Section~\ref{Sect:Renorm}, subsymmetry already provides a built-in structural rigidity, and this enables sharper control over the behavior of greedy and quasi-greedy bases. Our goal in this section is to show that, under mild regularity assumptions on the fundamental function, subsymmetric bases admit renormings that simultaneously enforce isometric subsymmetry and retain the greedy constants. This yields a clean counterpart to the results of the preceding sections and produces new examples of spaces with isometrically greedy bases.

To derive the subsymmetric counterpart to Lemma~\ref{lem:UncRen}, we need to tailor a renorming different from the one used in \cite{Ansorena2018}.

\begin{lemma}\label{lem:RNSub}
Let $\XB$ be a $C$-greedy basis of a Banach space $\XX$, $1\le C<\infty$. Assume that $\XB$ is subsymmetric. Then there is a renorming of $\XX$ so that $\XB$ $C$-greedy relative and isometrically subsymmetric relative to the renorming. Moreover, if $\psi$ and $\psi^*$ are respectively the fundamental functions of $\XB$ and $\XB^*$ relative to the renorming, then $\psi\le\usdf[\XB,\XX]$ and $\psi^*\le \usdf[\XB^*,\XX^*]$.
\end{lemma}

\begin{proof}
By Lemma~\ref{lem:UncRen}, we can assume that $\XX$ is already endowed with a norm relative to which $\XB$ is lattice $1$-unconditional. Let for each $k\in\NN$, $\Bt(k)$ be the set of all sequences $\beta\colon\NN\to \NN$ such that
\[
\beta(n)-\beta(n-1)\ge k, \quad n\in\NN.
\]
Clearly, $\Bt(k)$ decreases as $k$ increases. Note that $\Bt(1)$ is the set of all increasing functions from $\NN$ to $\NN$. Given a function $\lambda$ defined on a subset of $\NN$ and $\beta\in\Bt(1)$, we define $\beta(\lambda)$ by $k\mapsto \lambda(n)$ if $k=\beta(n)$ for some $n\in A$, and $k\mapsto 0$ otherwise. Define
\[
\norm{f}_s=\inf_{k\in\NN} \sup_{f\in\Bt(k)} \norm{L_\beta(f)}=\lim_{k\in\NN} \sup_{f\in\Bt(k)} \norm{L_\beta(f)}, \quad f\in\XX.
\]
Since $\XB$ is $C$-subsymmetric for some $C\in[1,\infty)$, $\norm{\cdot}_s$ is a renorming of $\XX$. Since
\[
L_\beta \circ M_\lambda=M_{\beta(\lambda)} \circ L_\beta
\]
for all $\lambda\in\ell_\infty$ and $\beta\in\Bt(1)$, $\XB$ is lattice $1$-unconditional relative to $\norm{\cdot}_a$. Since
\[
L_\beta\enpar{\Ind_{\varepsilon,A}[\XB,\XX]}=\Ind_{\beta(\varepsilon),\beta(A)}[\XB,\XX]
\]
for all $A\in\Pt_{<\infty}(\NN)$, all $\varepsilon\in\EE^A$ and all $\beta\in\Bt(1)$, $\psi\le \usdf[\XB,\XX]$. Since
\[
\Ind_{\varepsilon,A}[\XB,\XX^*](f)=\Ind_{\beta(\varepsilon),\beta(A)}[\XB,\XX^*]\enpar{L_\beta(f)},
\]
for all $A\in\Pt_{<\infty}(\NN)$, all $\varepsilon\in\EE^A$, $\beta\in\Bt(1)$ and all $f\in\XX$ we have
\[
\abs{\Ind_{\varepsilon,A}[\XB^*,\XX^*](f)} \le \usdf[\XB^*,\XX^*](m) \norm{f}_a.
\]

Let $\beta\in\Bt(1)$, $f\in \XX$, $A\in\Gt(f)$, $B\in\Pt(\Nt)$ with $\abs{B}\le\abs{A}$ and $(a_n)_{n\in B}$ in $\FF$. Set $g=\sum_{n\in A} a_n \, \xx_n$. Since $\beta(A)\in \Gt(L_\beta(f))$,
\begin{multline*}
\norm{L_\beta\enpar{f-S_B(f)}}
=\norm{L_\beta(f)-S_{\beta(A)}(f)}\\
\le C \norm {L_\beta(f)-L_\beta(g)}
=C\norm{L_\beta(f-g)}.
\end{multline*}
Hence, $\XB$ is $C$-greedy relative to the renorming.

It remains to prove that $\norm{\cdot}_a$ is indeed a $1$-subsymmetric renorming. To that end, we will use the following fact that needs no further explanation.
\begin{claim}\label{claim:SS}
Given $A\in\Pt_{<\infty}(\NN)$, $\alpha\colon A\to \NN$ one-to-one and order-preserving and $j\in\NN$, there is $k\in\NN$ such that for any $\beta\in\Bt(k)$ there is $\gamma\in\Bt(j)$ such that $\restr{\beta}{A}=\gamma\circ \alpha$.
\end{claim}

We have to prove that for any $A\in\Pt_{<\infty}(\NN)$, any $\alpha\colon A \to \NN$ one-to-one and order-preserving and any $(a_n)_{n\in A}$ in $\FF$, $\norm{f}_s\le \norm{g}_s$, where
\[
f=\sum_{n\in A} a_n \, \xx_n, \quad g=\sum_{n\in A} a_n \, \xx_{\alpha(n)}.
\]
Let $(\beta_k)_{k=1}^\infty$ be such that $\beta_k\in\Bt(k)$ for all $k\in\NN$ and
\[
\norm{f}_a=\lim_{k\in \NN} \norm{L_{\beta_k}(f)}.
\]
Use Claim~\ref{claim:SS} to recursevely pick $(\gamma_j)_{j=1}^\infty$ and $(k_j)_{j=1}^\infty$ such that $\gamma_j\in\Bt(j)$, $k_j\in \NN\cap[1+k_{j-1},\infty)$ and $\restr{\beta_{k_j}}{A}=\gamma_j \circ \alpha$ for all $j\in\NN$. Since $L_{\beta_{k_j}}(f)=L_{\gamma_j}(g)$ for all $j\in\NN$,
\[
\norm{f}_a
=\lim_{j\in \NN} \norm{L_{\beta_{k_j}}(f)}
=\lim_{j\in \NN} \norm{L_{\gamma_j}(g)} \le \norm{g}_a.\qedhere
\]
\end{proof}

Recall that symmetric bases are subsymmetric and that, in turn, subsymmetric bases are greedy. Garling \cite{Garling1968} provided the first-known example of a space with a subsymmetric basis that fails to be symmetric. Dilworth et al.\ \cite{DOSZ2011} found an isometrically greedy subsymmetric basis that fails to be symmetric, solving this way a problem posed in \cite{AW2006}. Later on, the authors of \cite{AAW2018b} proved that any Garling space (modelled after the above-mentioned example by Garling) can be renormed so that its canonical basis is isometrically greedy. Our construction yields the following contributions to the theory.

\begin{theorem}\label{thm:SubSym}
Let $\XX$ be a Banach space with a subsymmetric basis $\XB$. If $\usdf[\XB,\XX]$ has the LRP, then $\XX$ can be renormed so that $\XB$ is isometrically greedy, isometrically bidemocratic, and isometrically subsymmetric. Moreover, there is a Dini regularization $\prim$ of $\usdf[\XB,\XX]$ such that $\prim$ and $\prim^*$ are nondecreasing. Given such a sequence $\prim$ we can choose the renorming with respect to which the fundamental function of $\XB$ is $\prim$.
\end{theorem}

\begin{proof}
Just combine Theorem~\ref{thm:main}, Lemma~\ref{lem:RNSub}, and Theorem~\ref{thm:AW}.
\end{proof}

\begin{theorem}\label{thm:CTSS}
Let $\XB$ be a subsymmetric basis of a Banach space $\XX$ with nontrivial cotype. Then $\XX$ can be renormed so that $\XB$ is isometrically greedy, isometrically bidemocratic, and isometrically subsymmetric. Besides, there is a Dini regularization $\prim$ of $\usdf[\XB,\XX]$ such that $\prim$ and $\prim^*$ are nondecreasing. Given such a sequence $\prim$ we can choose the renorming so that the fundamental function of $\XB$ relative to it is $\prim$.
\end{theorem}

\begin{proof}
Just combine Lemma~\ref{lem:CTLRP} with Theorem~\ref{thm:SubSym}.
\end{proof}

\begin{corollary}
Let $\XB$ be a spreading model of a weakly null sequence in a Banach space $\YY$ with nontrivial cotype. Then the Banach space $\XX$ spanned by $\XB$ can be renormed so that $\XB$ is isometrically greedy, isometrically bidemocratic, and isometrically subsymmetric. Besides, there is a Dini regularization $\prim$ of $\usdf[\XB,\XX]$ such that $\prim$ and $\prim^*$ are nondecreasing. Given such a sequence $\prim$ we can choose the renorming so that the fundamental function of $\XB$ relative to it is $\prim$.
\end{corollary}

\begin{proof}
$\XB$ is a subsymmetric basis, and $\XX$ inherits the Rademacher cotype from $\YY$. Hence, Theorem~\ref{thm:CTSS} applies.
\end{proof}

Although Theorem~\ref{thm:CTSS} does not solve \cite{AAT2025}*{Problem 2.1}, it notably advances the knowledge of the behavior of subsymmetric bases with respect to the TGA. For instance, it yields the following.

\begin{theorem}
The Schlumprecht space $\Sym$ can be renormed so that its canonical basis is isometrically greedy, isometrically subsymmetric, and isometrically bidemocratic.
\end{theorem}

\begin{proof}
This important space which Schlumprecht invented to solve the distorsion problem (see \cite{Schlumprecht1991}), is built from a function $f\colon[1,\infty)\to[1,\infty)$ such that
\begin{itemize}
\item $f(1)=1$ and $f(x)<x$ for for $x\in(1,\infty)$,
\item $\lim_{x\to\infty} f(x)=\infty$ and $\lim_{x\to\infty} x^{-q} f(x)=0$ for all $q>0$,
\item $f(xy)\le f(x) f(y)$ for all $x$, $y\in[1,\infty)$, and
\item the dual function $x\mapsto x/f(x)$ is concave.
\end{itemize}
These properties imply that $\prim=\restr{f}{\NN}$ has the URP. The canonical basis $\St$ of $\Sym$ is subsymmetric, and its fundamental function is equivalent to $\prim^*$ (see \cite{Stankov2025}*{Fact 1}). By Theorem~\ref{thm:main}, there is a renorming of $\Sym$ such that $\St$ has the desired properties.
\end{proof}
\section{Renorming Banach spaces with conditional bases}\label{Sect:Condit}\noindent
Although separable Banach spaces without a Schauder basis do exist \cite{Enflo1973}, most Banach spaces encountered in analysis are naturally endowed with a Schauder basis. In some instances, this canonical basis is conditional, prompting the question of whether the space admits an unconditional basis. In others, the space is equipped with an unconditional basis, and one may instead ask whether a conditional Schauder basis can be constructed.

As regards the former problem, it is well known that neither $\Ct[0,1]$, nor $L_1$ \cite{LinPel1968}, nor the James space \cite{James1951} admits an unconditional basis. The latter problem traces back to the pioneering work of Babenko \cite{Babenko1948}, who constructed conditional bases for the separable infinite-dimensional Hilbert space. This line of research was later completed by Pe{\l}czy\'{n}ski and Singer \cite{PelSin1964}, who proved that every Banach space with a Schauder basis necessarily admits a conditional one.

More recently, motivated by the role played by unconditionality parameters in the performance of the Thresholding Greedy Algorithm (see, for instance, \cite{AAB2021}), the focus has shifted from the mere existence of conditional bases to the construction of bases whose unconditionality parameters exhibit prescribed growth. Building on Babenko's ideas, Garrig\'os and Wojtaszczyk \cite{GW2014} showed that $\ell_2$ admits, for each $0<\alpha<1$, a Schauder basis $\XB_\alpha$ satisfying
\[
\unc_m[\XB_\alpha,\ell_2]\gtrsim m^\alpha,
\qquad m\in\NN.
\]
They subsequently combined this construction with additional techniques to prove that $\XX=\ell_p$ and $\XX=L_p$, $1<p<\infty$, admit, for every $0<\alpha<1$, almost greedy bases $\YB_\alpha$ such
that
\[
\unc_m[\YB_\alpha,\XX]\gtrsim \log^\alpha(1+m),
\qquad m\in\NN.
\]
Further examples of highly conditional almost greedy bases for a broad class of Banach spaces were later obtained in \cite{AADK2019b}, using methods originating in \cite{DKK2003}.

The existence of a conditional basis satisfying Property~(A) remained open for many years (see \cite{AAT2025}*{Problem~6}). Although this question was recently answered in the affirmative in \cite{AABB2025}, one might still expect such examples to be rather scarce. As we show below, this intuition turns out to be misleading. Indeed, the renorming scheme developed in Section~\ref{sect:construction} allows us to link the existence of conditional bases with Property~(A) to the existence of almost greedy bases with prescribed unconditionality growth. In particular, we provide a mild condition ensuring that every Banach space with nontrivial type admits an \emph{isometrically almost greedy conditional basis} whose unconditionality parameters grow at a prescribed rate.

\begin{lemma}\label{lem:Doubling}
Let $f\colon\NN\to(0,\infty)$ be a nondecreasing unbounded sequence. Then there is a nondecreasing unbounded sequence $g\colon\NN\to(0,\infty)$ such that $g\le f$ and $g(2m)\lesssim g(m)$ for $m\in\NN$.
\end{lemma}

\begin{proof}
Pick $(C_m)_{m=1}^\infty$ in $(1,\infty)$ such that $\prod_{m=1}^\infty C_m=\infty$,
\[
D:=\sup_{k\in\NN} \prod_{m=k}^{2k-1} C_m<\infty
\]
and $C=\sup_m C_m<\infty$. We recursively define $g(1)=f(1)$ and
\[
g(m+1)=\min\enbrace{f(m+1),C_m g(m)}, \quad m\in\NN.
\]
By construction, $g\le f$. Hence, given $m\in\NN$,
\[
g(m+1)\ge \min\enbrace{f(m),C_m g(m)}\ge \min\enbrace{g(m),C_m g(m)}=g(m).
\]
Assume by contradiction that $G=\sup_m g(m)<\infty$. Pick $m_0\in\NN$ such that $f(m_0+1)\ge C G$. Then,
\[
g(m+1)\ge\min\enbrace{CG,C_m g(m)}=C_m g(m), \quad m\ge m_0.
\]
Therefore, $G=g(m_0)\prod_{m=m_0}^\infty C_m=\infty$. Finally, given $k\in\NN$,
\[
g(2k)\le g(k) \prod_{m=k}^{2k-1} C_m\le D g(k).\qedhere
\]
\end{proof}

\begin{theorem}
Let $\XX$ be a Banach space with a Schauder basis. Assume that $\XX$ has nontrivial type and contains a complemented subspace with a subsymmetric basis $\St$. Then $\XX$ can be renormed so that it has an isometrically almost greedy conditional basis $\XB$ with $\usdf[\XB,\XX]=\usdf[\St,\XX]$. Moreover, if $\St$ is equivalent to the standard $\ell_p$-basis for some $1<p<\infty$, given $0<\alpha<1$ we can get
\begin{equation}\label{eq:LCC}
\unc[\XB,\XX](m) \gtrsim \log^\alpha(1+m), \quad m\in\NN.
\end{equation}
\end{theorem}

\begin{proof}
It is known \cite{PelSin1964} that $\XX$ has a conditional Schauder basis. Hence, $\XX$ has a conditional almost greedy basis $\XB$ by Lemma~\ref{lem:Doubling} and \cite{AADK2019b}*{Theorem 4.1}. In the case when $\St$ is equivalent to the unit vector system of $\ell_p$, appealing to \cite{AADK2019b}*{Theorem 4.5} instead, we get that the almost greedy basis $\XB$ satisfies \eqref{eq:LCC}. By Lemma~\ref{lem:URPBD}, Lemma~\ref{lem:CTLRP}, Lemma~\ref{lem:TURP} and Theorem~\ref{thm:main}, we can renorm the space so that $\XB$ has Property~(A). We conclude the proof of the general statement by appealing to Theorem~\ref{thm:1AG}.
\end{proof}
\section{Open problems}\label{Sect:Prob}\noindent
The results obtained in this paper leave open several natural questions that we believe merit further investigation. Some of these arise from the renorming techniques developed here, while others touch on structural aspects of greedy-type bases that remain poorly understood. We gather below a selection of such problems, chosen both for their conceptual relevance and for the potential they offer in advancing the theory of nonlinear approximation and the geometry of Banach spaces.

Despite the fact that both $L_p$ and $L_{p'}$ can be renormed so that $\Ht_p$ and $\Ht_{p'}$ are isometrically greedy, it is unclear whether these norms are in duality.

\begin{question}
Let $1<p<\infty$. Is there a renorming $\Nnorm{\cdot}$ of $L_p$ so that $\Ht_p$ is isometrically greedy relative to $\Nnorm{\cdot}$ and $\Ht_{p'}$ is isometrically greedy relative to the dual norm of $\Nnorm{\cdot}$?
\end{question}

More generally, the existence of an isometrically greedy renorming for a basis $\XB$ of a Banach space $\XX$, and a renorming of $\XX^*$ relative to which $\XB^*$ is isometrically greedy does not guarantee the existence a renorming of $\XX$ so that both $\XB$ and $\XB^*$ are isometrically greedy.

In the absence of a general answer to the existence of isometrically greedy renormings for subsymmetric bases, we must pursue the study of subsymmetric bases whose fundamental function does not have the LRP. In particular, we pose the following.

\begin{question}
Can we renorm the dual space of the Schlumprecht space in such a way that its canonical basis becomes isometrically greedy?
\end{question}

Undoubtedly, the work of Dilworth et al.\@ \cite{DKOSZ2014} was the springboard to furthering the research towards the results achieved in this article. We take this opportunity to revisit an open problem raised there that is still open.

\begin{question}[\cite{DKOSZ2014}*{Problem C}]
Let $\XB$ be a greedy basis of a Banach space $\XX$, and $\epsilon>0$. Can we renorm $\XX$ so that $\XB$ is $(1+\epsilon)$-greedy?
\end{question}

Once we have shown that a large class of conditional almost greedy bases have Property~(A) under renorming, we should turn to bases where our construction does not apply. Arguably, the more important ones are the Lindenstrauss basis in $\ell_1$ (see \cite{DilworthMitra2001}) and the Haar system in $\BV (\RR^n)$, $n\ge 2$ (see \cite{Woj2003}). We emphasize that, unlike other conditional greedy-like bases, these bases originally emerged from settings outside approximation theory.

\begin{question}
Can $\ell_1$ be renormed so that the Lindenstrauss basis has Property~(A)?
\end{question}

\begin{question}
Can $\BV (\RR^n)$, $n\ge 2$, be renormed so that the Haar system has Property~(A)?
\end{question}
\section*{Statements and Declarations}
\subsection*{Conflict of interest}
The authors have no competing interests to declare that are relevant to the content of this article.
\subsection*{Data Availability} Since no datasets were generated or analyzed during the current study, data sharing does not apply to this article.

\end{document}